\documentclass{article}
\usepackage[utf8]{inputenc}
\usepackage{amsmath,amsthm,amssymb}
\usepackage{tikz} \usetikzlibrary{positioning,shapes.geometric}
\usepackage{graphicx}
\usepackage{hyperref}

\def\firstellip{(1.6, 0) ellipse [x radius=3cm, y radius=1.5cm, rotate=50]}
\def\secondellip{(0.3, 1cm) ellipse [x radius=3cm, y radius=1.5cm, rotate=50]} \def\thirdellip{(-1.6, 0) ellipse [x radius=3cm, y radius=1.5cm, rotate=-50]} \def\fourthellip{(-0.3, 1cm) ellipse [x radius=3cm, y radius=1.5cm, rotate=-50]} 

\newtheorem{thm}{Theorem}[section]

\newtheorem{lem}[thm]{Lemma}

\renewcommand{\Box}{\mathbin{\text{\scalebox{.84}{$\square$}}}}

\title{Cliques in realization graphs}
\author{Michael D. Barrus\\
Department of Mathematics and Applied Mathematical Sciences\\
University of Rhode Island\\
Kingston, RI 02881\\
\url{barrus@uri.edu}\\[12pt]
and\\[12pt]
Nathan Haronian\\
Department of Mathematics\\
Brown University\\
Providence, RI 02912\\
\url{nathan_haronian@brown.edu}
}
\date{}

\begin{document}

\maketitle

\begin{abstract}
    The realization graph $\mathcal{G}(d)$ of a degree sequence $d$ is the graph whose vertices are labeled realizations of $d$, where edges join realizations that differ by swapping a single pair of edges. Barrus [On realization graphs of degree sequences, Discrete Mathematics, vol. 339 (2016), no. 8, pp. 2146-2152] characterized $d$ for which $\mathcal{G}(d)$ is triangle-free. Here, for any $n \geq 4$, we describe a structure in realizations of $d$ that exactly determines whether $\mathcal{G}(d)$ has a clique of size $n$. As a consequence we determine the degree sequences $d$ for which $\mathcal{G}(d)$ is a complete graph on $n$ vertices.
\end{abstract}

\section{Introduction}
In this paper we discuss degree sequences of finite simple graphs. Such a degree sequence $d=(d_1,\dots,d_n)$ typically is realized by several graphs; here we consider these realizations as labeled graphs on a common vertex set $V=\{v_1,\dots,v_n\}$ in which the degree of vertex $v_i$ is necessarily $d_i$ for all $i \in \{1,\dots,n\}$.

It is natural to wonder about relationships between realizations of a degree sequence. One structure that encodes some of these relationships is the \emph{realization graph} $\mathcal{G}(d)$, which is the focus of this paper. In this graph the vertices are the labeled realizations of $d$. Any two vertices $H$ and $J$ are adjacent if the graphs $H$ and $J$ can be obtained from each other by a single modification of edge sets called a \emph{2-switch}, which we now define.

Given a graph $H$, an \emph{alternating 4-cycle} is a configuration involving four vertices $u,v,w,x$ in which $uv$ and $wx$ are edges and $ux$ and $vw$ are not edges in $H$. Representing non-edges by dotted lines, Figure~\ref{fig: A4} shows why this configuration has its name. Note that the definition does not impose any requirement about the ``diagonal'' vertex pairs $\{u,w\},\{v,x\}$. We denote such an alternating 4-cycle by $[u,v: w,x]$.
\begin{figure}
    \centering
    \includegraphics{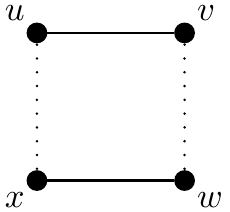}
    \caption{An alternating 4-cycle $[u,v: w,x]$.}
    \label{fig: A4}
\end{figure}

Suppose that a graph $H$ has degree sequence $d$. A \emph{2-switch} is an operation performed on an alternating 4-cycle $[u,v: w,x]$ in $H$: we delete the edges $uv,wx$ from the graph and add edges $ux,vw$. In this way the adjacencies between consecutive vertices in the alternating 4-cycle are each toggled, leaving an alternating 4-cycle $[v,w : x,u]$. Letting $J$ denote the graph after the 2-switch on $H$, observe that each vertex has the same degree in $J$ as in $H$. By our definition, $H$ and $J$ are adjacent in the realization graph $\mathcal{G}(d)$.

In this way the realization graph is the ``reconfiguration graph'' for the operation of a 2-switch on the realizations of a graph. See~\cite{Nishimura18} for survey of reconfiguration questions, of which there are many.

Figure~\ref{fig: G22211} displays an example of a realization graph. Here the graph shown is $\mathcal{G}((2,2,2,1,1))$, with the white vertex corresponding to the unique realization isomorphic to $K_3+K_2$, and the black vertices corresponding to the realizations isomorphic to a path. In this realization graph, the white vertex is adjacent to all the other vertices because for each of the six labeled path realizations, there is a 2-switch possible on the labeled $K_3+K_2$ that yields the given path.
\begin{figure}
    \centering
    \includegraphics[height=1.2in]{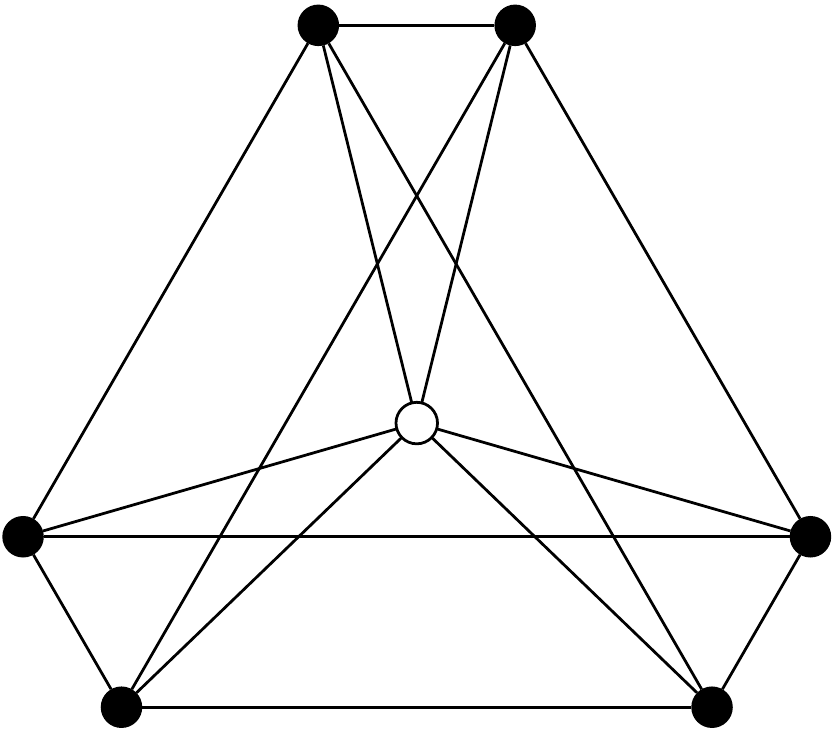}
    \caption{The realization graph $\mathcal{G}((2,2,2,1,1))$ and $\mathcal{G}((3,3,2,2,2))$}
    \label{fig: G22211}
\end{figure}

A classic result discovered or hinted at independently by many authors (for example, see~\cite{FHM65,Hakimi63,Petersen91,Senior51}) states that any two labeled graphs with the same degree sequence have the property that one can be iteratively transformed into the other by a finite sequence of 2-switches. This implies that $\mathcal{G}(d)$ is connected for all $d$.

Another simple result concerns complements. The graph in Figure~\ref{fig: G22211} is also the realization graph of $\mathcal{G}((3,3,2,2,2))$. This is because $(2,2,2,1,1)$ and $(3,3,2,2,2)$ are degree sequences of graphs that are complements of each other. In general, when the complement of a graph is taken, an alternating 4-cycle $[u,v:w,x]$ gives rise to an alternating 4-cycle $[v,w:x,u]$ in the resulting graph, and 2-switches performed on these alternating 4-cycles produce graphs that are again complementary. For this reason, if realizations $H$ and $J$ of a degree sequence $d$ are adjacent in $\mathcal{G}(d)$, then the complements of $H$ and $J$ will be adjacent in the realization graph of their ``complementary'' degree sequence. It follows that the degree sequences $d=(d_1,\dots,d_n)$ and $\overline{d}=(n-1-d_n,\dots,n-1-d_1)$ have the same realization graph, up to isomorphism.

Perhaps of the earliest mention of realization graphs of degree sequences appears in the paper~\cite{EggletonHolton79} by Eggleton and Holton. (Around the same time, Brualdi~\cite{Brualdi80} introduced the \emph{interchange graphs} for 0-1 matrices with prescribed row and column sums; Arikati and Peled~\cite{ArikatiPeled99} noted that realization graphs of degree sequences of split graphs are equivalent to interchange graphs of suitably chosen matrices.) In~\cite{ArikatiPeled99}, the question is raised of whether realization graphs all have a hamiltonian path or cycle; at present this is still an open question.

In~\cite{Barrus16}, Barrus showed that the realization graph $\mathcal{G}(d)$ is the Cartesian product of the realization graphs of the degree sequences that make up $d$ in a decomposition due to Tyshkevich~\cite{Tyshkevich00}. 

To preface the main question of this paper, we recall some definitions and a result. A \emph{clique} in a graph is a set of vertices that are pairwise adjacent, and a \emph{triangle} is a complete subgraph having three vertices. In~\cite{Barrus16}, Barrus touched on the notion of small cliques in realization graphs by characterizing the triangle-free realization graphs $\mathcal{G}(d)$ and the corresponding degree sequences $d$. Restating part of the analysis there, we have the next theorem. Here a \emph{configuration} refers to a triple $(W,F,F')$ where $W$ is a vertex set and $F$ and $F'$ are disjoint sets of pairs $\{u,v\}$ where $u,v \in W$. For a graph $H$ to \emph{contain} a configuration $(W,F,F')$ means that there exists an injective map $f:W \to V(H)$ carrying elements of $F$ to edges of $H$ and elements of $F'$ to non-edges in $H$.

\begin{thm}[\cite{Barrus16}, Theorem 9] \label{thm: triangles}
    For any degree sequence $d$ and realization $H$ of $d$, the vertex $H$ belongs to a triangle in $\mathcal{G}(d)$ if and only if $H$ contains $2K_2$ or $C_4$ as an induced subgraph or contains the configuration shown in Figure~\ref{fig: matrogenic}.
\end{thm}

\begin{figure}
    \centering
    \includegraphics[height=0.75in]{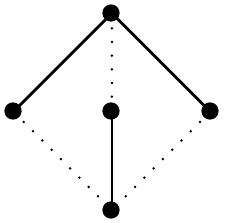}
    \caption{A configuration leading to a triangle in realization graphs}
    \label{fig: matrogenic}
\end{figure}

Theorem~\ref{thm: triangles} suggests further exploration. To have cliques larger than a triangle appear in a realization graph, a large collection of distinct realizations of a degree sequence must differ in their edge sets, but only slightly, so that each differs from any other by a single 2-switch. How can this be achieved? Here, if the clique size is a large integer $q$ and $H$ is a realization forming a vertex in the clique, then there must be distinct alternating cycles in $H$ that allow for the transformation of $H$ into each of the other $q-1$ realizations comprising the clique. Furthermore, each of the resulting $q-1$ realizations must be reachable from any other via a single 2-switch. Is this possible? If so, what structures in $H$ are necessary or sufficient for this to happen?

We will present a generalization of Theorem~\ref{thm: triangles} that answers these questions for cliques of any size; the full statement appears in Theorem~\ref{thm: iff} after necessary definitions and concepts are introduced. Given $q \geq 2$, we present a certain subgraph in Section~2 whose presence in any realization $H$ of $d$ leads to the inclusion of $H$ in a clique of size $q$ in $\mathcal{G}(d)$. Then, in Section~3, we show that this construction is always present in realizations belonging to cliques of order at least 4, so we obtain a characterization extending Theorem~\ref{thm: triangles}. Finally, in Section~4 we identify the degree sequences whose realization graphs are complete graphs; Theorem~\ref{thm: Kn iff} presents the characterization.

We establish a few items of notation and definition. In this paper a degree sequence is represented as an ordered list of integers, typically written in nonincreasing order. In a degree sequence, let $t^{(k)}$ denote the appearance of $t$ as a term $k$ distinct times; hence the degree sequence of the graph in Figure~\ref{fig: G22211} may be written as $(6,4^{(6)})$. A complete graph on $n$ vertices, i.e., a graph in which each possible pair of its $n$ vertices is adjacent, will be denoted by $K_n$. An \emph{independent set} will be a set of vertices that are pairwise nonadjacent. The disjoint union of two graphs $G$ and $H$ will be denoted by $G+H$, and the disjoint union of $t$ copies of the same graph $G$ will be written as $tG$. Finally, we use $\overline{G}$ to denote the complement of a graph $G$, i.e., the graph having the same vertex set as $G$ in which two vertices are adjacent precisely if they are not adjacent in $G$.

\section{A structure producing cliques in $\mathcal{G}(d)$} \label{sec: dials}

In this section we present a structure that can appear among the realizations of a degree sequence to produce a clique of any size. Visually, it bears some resemblance to an analog dial and needle (see Figure~\ref{fig: analog}), which motivates the name we give it.

\begin{figure}
    \centering
    \includegraphics[height=0.75in]{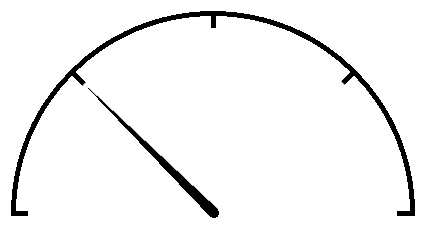}
    \caption{An analog dial and needle}
    \label{fig: analog}
\end{figure}

Given a set $\mathcal{S} = \{R_1,\dots,R_n\}$ of labeled realizations of the same degree sequence having the same vertex set $V$, define a \emph{dial with respect to $\mathcal{S}$} to be a pair of sets $(W,P)$ satisfying the following conditions.
\begin{enumerate}
	\item[(a)] The second entry $P$ is the set of all pairs of vertices from $V$ that differ in their status (adjacent or non-adjacent) among $R_1,\dots,R_n$. More precisely, for $a,b \in V$, the pair $\{a,b\}$ will belong to $P$ if $ab$ is an edge in some $R_i$ and not an edge in some $R_j$, where $i,j \in \{1,\dots,n\}$. The set $W$ is the union of all pairs in $P$, so $W \subseteq V$.
	\item[(b)] There exist two vertices $u,v \in W$ such that for every vertex $w \in W \setminus \{u,v\}$, both the pairs $\{u,w\}$ and $\{v,w\}$ belong to $P$, and no other pair belongs to $P$.
	\item[(c)] In every realization $R_i$ for $i \in \{1,\dots,n\}$, vertex $u$ is adjacent to exactly one vertex, denoted $w_i$, in $W \setminus \{u,v\}$. (This edge $uw_i$ is called the \emph{needle} in $R_i$.) In the same realization $R_i$, the vertex $v$ is not adjacent to $w_i$ but is adjacent to every vertex in $W \setminus \{u,v,w_i\}$.
\end{enumerate}

Given a dial with respect to $\mathcal{S}$, the induced subgraph in any $R_i$ having vertex set $W$ is called a \emph{dial state}. Within each $R_i$, the vertex set $W$ and the edges and non-edges from $P$ form a \emph{dial configuration}. Ignoring vertex labels, let $\mathcal{D}_n$ denote an unlabeled configuration of $n+2$ vertices, $n$ edges, and $n$ non-edges arranged as in a dial configuration. With this notation, the configuration in Figure~\ref{fig: matrogenic} is hence denoted $\mathcal{D}_3$.

In Figure~\ref{fig: 4 dials} we illustrate the dial configurations in four graphs $R_1,R_2,R_3,R_4$, using dotted segments to indicate non-adjacencies; each is an instance of $\mathcal{D}_4$. In each configuration the top vertex is $v$, the bottom vertex is $u$, and the middle vertices are $w_1,w_2,w_3,w_4$. We emphasize that $u$, $v$, and the interior vertices in each configuration are the same vertices in each realization; the only thing that varies in each configuration (or in each dial state) is which pair in $P$ containing $u$ is the needle.

\begin{figure}
    \centering
    \includegraphics[width=\textwidth]{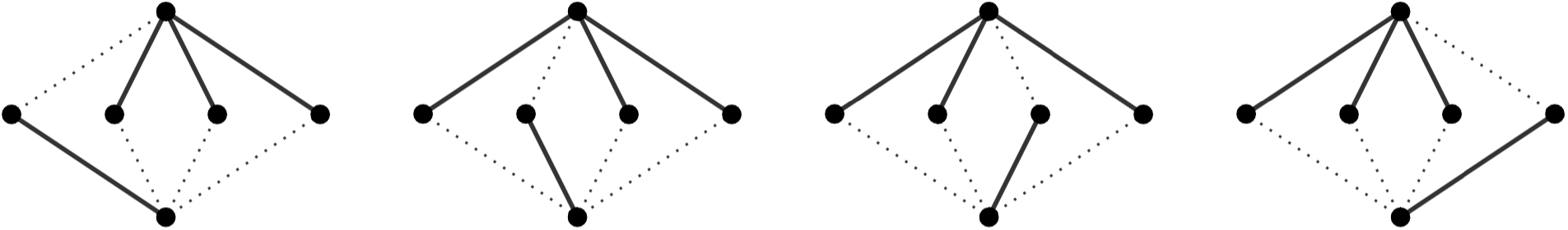}
    \caption{Configurations from states of a dial in four realizations of a degree sequence}
    \label{fig: 4 dials}
\end{figure}

\begin{lem} \label{lem: dial implies clique}
    If a dial exists for a set $\{R_1,\dots,R_n\}$ of realizations of a degree sequence $d$, then these realizations form a clique in the realization graph $\mathcal{G}(d)$.
    
    Furthermore, if some realization $R$ of a degree sequence $d$ contains the configuration $\mathcal{D}_n$, then $R$ belongs to a clique of size $n$ in $\mathcal{G}(d)$.
\end{lem}
\begin{proof}
    Given the dial for $\{R_1,\dots,R_n\}$ as indicated, let $u$, $v$, and $w_1,\dots,w_n$ denote the vertices of the dial as described above. For any $i,j$ in $\{1,\dots,n\}$, the 2-switch on graph $R_i$ using alternating 4-cycle $[u,w_i:v,w_j]$ produces the graph $R_j$. Hence these realizations are pairwise adjacent in $\mathcal{G}(d)$.
    
    Suppose now that some realization $R$ of a degree sequence $d$ contains the configuration $\mathcal{D}_n$. The $n-1$ alternating 4-cycles that use edges and non-edges from this configuration and include ``the needle'' permit 2-switches yielding $n-1$ additional, distinct realizations of $d$. It is straightforward to see that the $n+2$ vertices involved form the vertex set of a dial for these realizations, so as before $R$ belongs to a clique of size $n$ in $\mathcal{G}(d)$.
\end{proof}
In~\cite{FoldesHammer78}, F\"{o}ldes and Hammer characterized \emph{matrogenic graphs} as those for which no five vertices' adjacency relationships admitted the configuration $\mathcal{D}_3$ from Figure~\ref{fig: matrogenic}. As an immediate corollary to Lemma~\ref{lem: dial implies clique}, we conclude that every non-matrogenic graph is a vertex in a triangle in the realization graph of its degree sequence. (With some additional conditions, the reverse implication is true; for details, see~\cite{Barrus16}.)

\section{Necessity of the construction}
In this section we prove a near converse to Lemma~\ref{lem: dial implies clique}. 

\begin{thm} \label{thm: clique implies dial}
    If $R_1,\dots,R_n$ are the vertices of a clique in a realization graph $\mathcal{G}(d)$, where $n \geq 4$, then a dial exists for this collection of graphs. Moreover, the corresponding dial configuration in each realization $R_i$ contains all alternating 4-cycles necessary for 2-switches converting $R_i$ into $R_j$ for $j \in \{1,\dots,n\}\setminus\{i\}$.
\end{thm}

Observe that if $n=2$ in the hypothesis above, then the conclusion is still valid and follows from the definition of $\mathcal{G}(d)$; the states of the dial are simply the ``before'' and ``after'' versions of the alternating 4-cycle on which the 2-switch is performed. The conclusion in Theorem~\ref{thm: clique implies dial} does not hold for $n=3$, however; for instance, the three realizations of $(1,1,1,1)$ form a triangle in the realization graph though none contains the configuration $\mathcal{D}_3$. A similar result is true for many graphs containing an an induced subgraph with degree sequence $(1,1,1,1)$ or a chordless cycle on 4 vertices (in which case the graph's complement contains the induced subgraph). Note that these examples are mentioned along with $\mathcal{D}_3$ in Theorem~\ref{thm: triangles}.

We prove Theorem~\ref{thm: clique implies dial} for the cases $n \geq 4$ by induction. Section~\ref{subsec: base case} contains the result for $n=4$, and Section~\ref{subsec: induction step} contains the induction step.

\subsection{Base case} \label{subsec: base case}

Let $R_1,R_2,R_3,R_4$ be the vertices of a clique of size 4 in some realization graph $\mathcal{G}(d)$. Let $m$ be the number of edges in each realization. Since these four graphs are a clique in $\mathcal{G}(d)$, for each pair $i,j$ of distinct elements in $\{1,2,3,4\}$, the graph $R_i$ can be transformed into $R_j$ by a single 2-switch. This requires that $R_i$ and $R_j$ share $m-2$ edges and that each contain two edges that the other does not.

To analyze these requirements, we let $s_{I}$ denote the number of edges that appear in every realization $R_i$ for $i$ displayed in the subscript $I$ and that do \emph{not} appear in any realization $R_j$ for $j$ not displayed in $I$. Here the subscripts $I$ correspond to subsets of $\{1,2,3,4\}$ (written without enclosing braces or commas). Using a Venn diagram whose ellipses respectively represent the edge sets of $R_1,R_2,R_3,R_4$, the variables $s_I$ in the interior regions of Figure~\ref{fig: Venn} indicate the sizes of the subsets to which the various regions correspond.

\begin{figure}
    \centering
    \begin{tikzpicture} 
        \draw \firstellip node [label={[xshift=2.1cm, yshift=-1cm]$E(R_1)$}] {};
        \draw \secondellip node [label={[xshift=2.4cm, yshift=2.1cm]$E(R_2)$}] {};
        \draw \thirdellip node [label={[xshift=-2.1cm, yshift=-1cm]$E(R_3)$}] {};
        \draw \fourthellip node [label={[xshift=-2.4cm, yshift=2.1cm]$E(R_4)$}] {};
        \draw (0.4,-1.2)--(2.6,-1.6);
        \draw (-0.4,-1.2)--(-2.6,-1.6);
        \node at (2.6,0.3) {$s_{1}$};
        \node at (1.3,2.6) {$s_{2}$};
        \node at (-2.6,0.3) {$s_{3}$};
        \node at (-1.3,2.6) {$s_{4}$};
        \node at (1.8,1.5) {$s_{12}$};
        \node at (0,-1.8) {$s_{13}$};
        \node at (1.3,-0.75) {$s_{14}$};
        \node at (-1.3,-0.75) {$s_{23}$};
        \node at (0,1.5) {$s_{24}$};
        \node at (-1.8,1.5) {$s_{34}$};
        \node at (-3,-1.7) {$s_{123}$};
        \node at (0.8,0.3) {$s_{124}$};
        \node at (3,-1.7) {$s_{134}$};
        \node at (-0.8,0.3) {$s_{234}$};
        \node at (0,-0.7) {$s_{1234}$};
    \end{tikzpicture} 
    \caption{Venn diagram of the edges sets of $R_1,R_2,R_3,R_4$, with cardinalities indicated}
    \label{fig: Venn}
\end{figure}
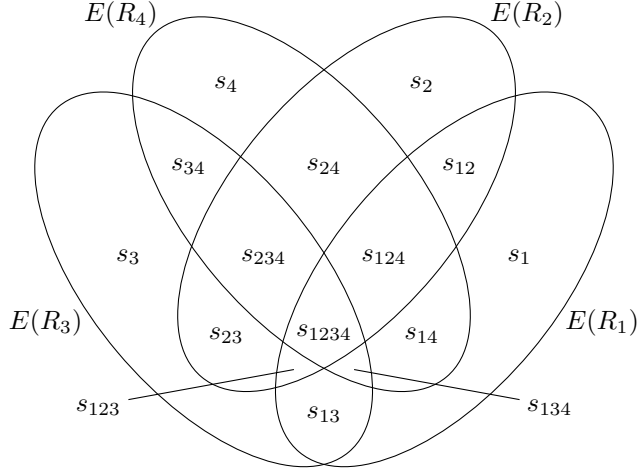

We use these variables to describe the overlaps in our four pairwise-adjacent realizations, obtaining the following system of equations. 
\begin{align}
    \sum_{I \ni i} s_I &= m \qquad \text{for $1 \leq i \leq 4$;} \label{eq: edges e}\\
    \sum_{\substack{J \ni i\\J \not \ni j}} s_J &= 2 \qquad \text{for $1 \leq i < j \leq 4$.} \label{eq: differences}
\end{align}
Here \eqref{eq: edges e} holds because $R_i$ has exactly $m$ edges. The equations in~\eqref{eq: differences} model the fact that $R_i$ has exactly two edges that $R_j$ does not, as mentioned above; as we will see shortly, the condition $i<j$ ensures that the overall system satisfies no linear dependence relations.

Using these equations, we construct a 10-by-16 augmented matrix $M$ for the system, which we display below followed by its reduced echelon form $M'$. Here the first 15 matrix columns are indexed by the subscripts on the corresponding variables $s_I$, with the variables $s_i$ first, ordered lexicographically, followed by the variables $s_{ij}$, ordered lexicographically, followed by the variables $s_{ijk}$, in \emph{reverse} lexicographic order, and followed finally by the variable $s_{1234}$.
\setcounter{MaxMatrixCols}{20}
\[M = 
\begin{bmatrix}
1 & 0 & 0 & 0 & 1 & 1 & 1 & 0 & 0 & 0 & 0 & 1 & 1 & 1 & 1 & m\\ 
0 & 1 & 0 & 0 & 1 & 0 & 0 & 1 & 1 & 0 & 1 & 0 & 1 & 1 & 1 & m\\ 
0 & 0 & 1 & 0 & 0 & 1 & 0 & 1 & 0 & 1 & 1 & 1 & 0 & 1 & 1 & m\\ 
0 & 0 & 0 & 1 & 0 & 0 & 1 & 0 & 1 & 1 & 1 & 1 & 1 & 0 & 1 & m\\ 
1 & 0 & 0 & 0 & 0 & 1 & 1 & 0 & 0 & 0 & 0 & 1 & 0 & 0 & 0 & 2\\ 
1 & 0 & 0 & 0 & 1 & 0 & 1 & 0 & 0 & 0 & 0 & 0 & 1 & 0 & 0 & 2\\ 
1 & 0 & 0 & 0 & 1 & 1 & 0 & 0 & 0 & 0 & 0 & 0 & 0 & 1 & 0 & 2\\ 
0 & 1 & 0 & 0 & 1 & 0 & 0 & 0 & 1 & 0 & 0 & 0 & 1 & 0 & 0 & 2\\ 
0 & 1 & 0 & 0 & 1 & 0 & 0 & 1 & 0 & 0 & 0 & 0 & 0 & 1 & 0 & 2\\ 
0 & 0 & 1 & 0 & 0 & 1 & 0 & 1 & 0 & 0 & 0 & 0 & 0 & 1 & 0 & 2 
\end{bmatrix};
\]
\[M' = 
\begin{bmatrix}
1 & 0 & 0 & 0 & 0 & 0 & 0 & 0 & 0 & 0 & 0 & -1 & -1 & -1 & -2 & 6-2m\\ 
0 & 1 & 0 & 0 & 0 & 0 & 0 & 0 & 0 & 0 & -1 & 0 & -1 & -1 & -2 & 6-2m\\ 
0 & 0 & 1 & 0 & 0 & 0 & 0 & 0 & 0 & 0 & -1 & -1 & 0 & -1 & -2 & 6-2m\\ 
0 & 0 & 0 & 1 & 0 & 0 & 0 & 0 & 0 & 0 & -1 & -1 & -1 & 0 & -2 & 6-2m\\ 
0 & 0 & 0 & 0 & 1 & 0 & 0 & 0 & 0 & 0 & 0 & 0 & 1 & 1 & 1 & m-2\\ 
0 & 0 & 0 & 0 & 0 & 1 & 0 & 0 & 0 & 0 & 0 & 1 & 0 & 1 & 1 & m-2\\ 
0 & 0 & 0 & 0 & 0 & 0 & 1 & 0 & 0 & 0 & 0 & 1 & 1 & 0 & 1 & m-2\\ 
0 & 0 & 0 & 0 & 0 & 0 & 0 & 1 & 0 & 0 & 1 & 0 & 0 & 1 & 1 & m-2\\ 
0 & 0 & 0 & 0 & 0 & 0 & 0 & 0 & 1 & 0 & 1 & 0 & 1 & 0 & 1 & m-2\\ 
0 & 0 & 0 & 0 & 0 & 0 & 0 & 0 & 0 & 1 & 1 & 1 & 0 & 0 & 1 & m-2 
\end{bmatrix}.
\]

Having constrained the values of the variables $s_I$ by the system in~\eqref{eq: edges e} and~\eqref{eq: differences}, we may further restrict the possible values for these  variables with a few lemmas.

\begin{lem} \label{lem: vars at most 1}
	If $I \subset \{1,2,3,4\}$ with $I \neq \emptyset$ and $I \neq \{1,2,3,4\}$, then $s_I \leq 1$.
\end{lem}
\begin{proof}
	Suppose to the contrary that $s_I \geq 2$ for some $I$ as described.
	
	Consider the case $|I|=1$ first. Re-indexing if necessary, we may assume that $I=\{1\}$. Taking $i=1$ and $j=2$ in \eqref{eq: differences} above, we have $s_I \leq 2$. However, if $s_I=2$, then three distinct alternating 4-cycles (those used in 2-switches changing $R_1$ to each of $R_2$, $R_3$, and $R_4$) would use the same pair of edges, which is impossible. Thus $s_I \leq 1$ if $|I|=1$.
	
	We may apply this same argument to the complementary realizations $\overline{R_1}$, $\overline{R_2}$, $\overline{R_3}$, and  $\overline{R_4}$, which form a clique in the realization graph of their collective degree sequence. Any edge appears in exactly one of these realizations $\overline{R_i}$ if and only if it is an edge in each graph $R_j$ for $j \in \{1,2,3,4\} \setminus \{i\}$. It follows that $s_I \leq 1$ if $|I| = 3$ as well.
	
	Supposing now that $|I|=2$, by re-indexing if necessary we may assume that $I=\{1,2\}$ and that $s_{12} \geq 2$. As before, \eqref{eq: differences} yields $s_{12}\leq 2$, so $s_{12}=2$. Let $uv,wx$ be these two edges in $R_1$. Since $R_3$ and $R_4$ have distinct edge sets, the 2-switches changing $R_1$ into each must differ on which non-edges are involved in the corresponding alternating 4-cycles (since both contain $uv,wx$). Without loss of generality we may assume that the 2-switch changing $R_1$ into $R_3$ uses edges non-edges $ux,vw$, and that the 2-switch changing $R_1$ into $R_4$ uses non-edges $uw,vx$. This requires that the subgraph of $R_1$ induced by $\{u,v,w,x\}$ be isomorphic to $2K_2$; the subgraph of $R_3$ on these vertices must be as well. 
	Note that the edges $uv,wx$ are present in $R_2$ but not in $R_3$, so the alternating 4-cycle used in the 2-switch transforming $R_3$ into $R_2$ must include non-edges $uv,wx$ from $R_3$. However, the only edges in $R_3$ induced by the vertex set $\{u,v,w,x\}$ are the edges $ux,vw$, and if we use these edges together with the requisite non-edges in a 2-switch, instead of creating $R_2$ we in effect undo the previous 2-switch, recreating $R_1$, a contradiction.
	
	Hence $s_I \leq 1$ for all sets $I \subseteq \{1,2,3,4\}$ satisfying $1 \leq |I| \leq 3$.
\end{proof}

From the reduced augmented matrix $M'$ we see that solutions to the system in \eqref{eq: edges e} and~\eqref{eq: differences} are determined by the value of five of the variables $s_I$. Lemma~\ref{lem: vars at most 1} implies that each variables $s_I$, other than $s_{1234}$, equals either 0 or 1. 
There are 32 solutions of the system produced by substituting candidate values for $s_{234},s_{134},s_{124},s_{123}$, and $s_4$; in only ten is every variable a nonnegative integer (and equal to 0 or 1 if the variable is not $s_{1234}$). We display these here, with one solution per line: 

\begin{center}
{\scriptsize
    \begin{tabular}{ccccccccccccccc}
        $s_1$ & $s_2$ & $s_3$ & $s_4$ & $s_{12}$ &  $s_{13}$ & $s_{14}$ & $s_{23}$ & $s_{24}$ & $s_{34}$ & $s_{234}$ & $s_{134}$ & $s_{124}$ & $s_{123}$ & $s_{1234}$\\ \hline
        0 & 0 & 0 & 0 & 1 & 1 & 1 & 1 & 1 & 1 & 0 & 0 & 0 & 0 & $m-3$\\ \hline
        0 & 1 & 1 & 1 & 1 & 1 & 1 & 0 & 0 & 0 & 1 & 0 & 0 & 0 & $m-3$\\ 
        1 & 0 & 1 & 1 & 1 & 0 & 0 & 1 & 1 & 0 & 0 & 1 & 0 & 0 & $m-3$\\
        1 & 1 & 0 & 1 & 0 & 1 & 0 & 1 & 0 & 1 & 0 & 0 & 1 & 0 & $m-3$\\
        1 & 1 & 1 & 0 & 0 & 0 & 1 & 0 & 1 & 1 & 0 & 0 & 0 & 1 & $m-3$ \\ \hline
        1 & 0 & 0 & 0 & 0 & 0 & 0 & 1 & 1 & 1 & 0 & 1 & 1 & 1 & $m-4$\\
        0 & 1 & 0 & 0 & 0 & 1 & 1 & 0 & 0 & 1 & 1 & 0 & 1 & 1 & $m-4$\\
        0 & 0 & 1 & 0 & 1 & 0 & 1 & 0 & 1 & 0 & 1 & 1 & 0 & 1 & $m-4$ \\
        0 & 0 & 0 & 1 & 1 & 1 & 0 & 1 & 0 & 0 & 1 & 1 & 1 & 0 & $m-4$\\ \hline
        1 & 1 & 1 & 1 & 0 & 0 & 0 & 0 & 0 & 0 & 1 & 1 & 1 & 1 & $m-4$
    \end{tabular}
    }
\end{center}
(In the table we have used horizontal lines to group solutions that are equivalent up to permuting the names of the realizations $R_1,R_2,R_3,R_4$.)

Though Lemma~\ref{lem: vars at most 1} considerably narrowed the possibilities for our candidate values for the variables $s_I$, even among the ten settings we have found, not all of them actually reflect a possible situation for the realizations $R_1,R_2,R_3,R_4$. Our next lemma will rule out all possibilities but one.

\begin{lem} \label{lem: No 12,13}
	Suppose that $A=\{i,j\}$ and $B=\{i,k\}$ for distinct elements $i,j,k$ from $\{1,2,3,4\}$. Then $s_A + s_B \leq 1$.
\end{lem}
\begin{proof}
	Suppose to the contrary that $s_A+s_B > 1$ for some sets $A,B$ as described; by Lemma~\ref{lem: vars at most 1} this implies that $s_A = s_B = 1$. By re-indexing the realizations as necessary, we may suppose that $A=\{1,2\}$ and $B=\{1,3\}$.
	
	Now $R_1$ and $R_2$ have an edge $e_{12}$ that does not appear in $R_3$ or $R_4$. Likewise, $R_1$ and $R_3$ have an edge $e_{13}$ that does not appear in $R_2$ or $R_4$; hence $e_{13}$ is distinct from $e_{12}$. The 2-switch transforming $R_1$ to $R_4$ must remove both edges $e_{12}$ and $e_{13}$; since these edges must appear in the corresponding alternating 4-cycle in $R_1$, $e_{12}$ and $e_{13}$ have no vertex in common.
	
	However, consider the 2-switch transforming $R_2$ into $R_3$. The corresponding alternating 4-cycle in $R_2$ must include the edge $e_{12}$ and the non-edge $e_{13}$. This requires that $e_{12}$ and $e_{13}$ share a vertex, which we showed above is not true. The contradiction shows that for any sets $A$ and $B$ satisfying the conditions in this lemma, we have $s_A+s_B \leq 1$.	
\end{proof}

Observe that in each of the first nine rows of the table above we find indices $i,j,k$ such that $s_{ij} = s_{ik}=1$, contradicting Lemma~\ref{lem: No 12,13}. Hence the last row must describe the edges of $R_1,R_2,R_3,R_4$; we have $s_i = 1$, $s_{ij}=0$, and $s_{ijk}=1$ for all distinct $i,j,k \in \{1,2,3,4\}$. 

Since $s_i=1$ for all $i$ and $s_{I}=1$ where $I$ consists of the three elements in $\{1,2,3,4\} \setminus \{i\}$, each realization $R_i$ has exactly one edge $e_i$ that none of the other three realizations has, and exactly one non-edge $f_i$ that all of the other three realizations have. Think now of the 2-switches transforming $R_1$ into each of $R_2,R_3,R_4$. Each of these 2-switches must toggle both the edge $e_1$ and the non-edge $f_1$. It follows that $e_1$ and $f_1$ share a vertex, and taking the union of the vertex sets, edge sets, and non-edge sets of the alternating 4-cycle configurations involved in these three 2-switches results in a configuration $\mathcal{D}_4$ in $R_1$, since no two of the alternating 4-cycles can agree on the fourth vertex while still being distinct from each other. (In fact, the configuration's respective appearances in $R_1,R_2,R_3,R_4$ are the same as those illustrated in Figure~\ref{fig: 4 dials}.) The six vertices involved are the vertices of a dial with respect to $\{R_1,R_2,R_3,R_4\}$ (here the edge $f_i$ is the needle in $R_i$, for each $i$), and we have established the base case in our inductive proof of Theorem~\ref{thm: clique implies dial}.

\subsection{Induction step} \label{subsec: induction step}
Suppose that the conclusion in Theorem~\ref{thm: clique implies dial} holds for cliques of size $k$ in every realization graph, for some $k \geq 4$. In this section we complete the induction by proving that every clique of size $k+1$ in any realization graph corresponds to the existence of a dial with respect to the realizations in the clique.

Let $\mathcal{G}(d)$ be an arbitrary realization graph having a clique of size $k+1$, and let $R_1,\dots,R_{k+1}$ be the vertices of the clique. Applying the induction hypothesis to $R_1,\dots,R_k$, we let $u,v,w_1,\dots,w_k$ be the vertices of the dial $(W,P)$ for these graphs, assuming that $uw_i$ is the needle in $R_i$ for each $i \in \{1,\dots,k\}$. 

If we apply the induction hypothesis to $R_2,\dots,R_{k+1}$, we arrive at a dial $(W',P')$ for these graphs as well. From the first dial we note that only $u$ and $v$ appear in each of the alternating 4-cycles used for 2-switches among $R_2,R_3,R_4$. Since these alternating 4-cycles must appear in the appropriate states of the second dial, the vertices $u$ and $v$ fulfill the same roles in the second dial that they do in the first: $u$ is the vertex common to every needle edge in the second dial's states, and $v$ is the other vertex common to every alternating 4-cycle used for 2-switches among $\{R_2,\dots,R_{k+1}\}$. Similarly, the edges $uw_2,\dots,uw_{k}$ are the needles for the graphs $R_2,\dots,R_k$ in the second dial as well as the first. Hence the symmetric difference of $P$ and $P'$ is \[\{\{u,w_1\}, \{v,w_1\}, \{u,w_{k+1}\},\{v,w_{k+1}\}\},\] where $w_{k+1}$ is the unique vertex in $W' \setminus W$; note that we may assume that $w_{k+1} \neq w_1$, since otherwise $R_1=R_{k+1}$, a contradiction.

From the first dial we see that in each of $R_2,\dots,R_k$, vertex $w_1$ is adjacent to $v$ and not to $u$. The 2-switch changing $R_2$ to $R_{k+1}$ does not change the neighbors of $w_1$, so $vw_1$ is an edge and $uw_1$ is a non-edge in $R_{k+1}$. A similar argument about the vertex $w_{k+1}$ shows that the pair $(W \cup W',P \cup P')$ is a dial for $R_1,\dots,R_{k+1}$, and our proof of Theorem~\ref{thm: clique implies dial} is complete.

\section{Conclusion}

Combining Lemma~\ref{lem: dial implies clique} and Theorem~\ref{thm: clique implies dial}, we have shown the following.

\begin{thm} \label{thm: iff} 
Let $d$ be a degree sequence, and let $R$ be a realization of $d$; also let $n \geq 4$. In the realization graph $\mathcal{G}(d)$ the vertex $R$ belongs to a clique of size $n$ if and only if $R$ contains the configuration $\mathcal{D}_n$.

Furthermore, moving in $\mathcal{G}(d)$ from $R$ to another vertex of the clique corresponds precisely to performing a 2-switch using edges and non-edges of the configuration $\mathcal{D}_n$ in $R$.
\end{thm}

In Section~1 we described the seeming potential difficulty in having several labeled realizations be pairwise adjacent in a realization graph. It is perhaps not surprising that Theorem~\ref{thm: iff} shows that this can happen in only one way.

In this section we conclude our results by characterizing the degree sequences $d$ for which $\mathcal{G}(d)$ is a complete graph. It will turn out that there is only ``one way'' in which this can happen as well; however, this claim is subject to our observation in Section 1 that complementary degree sequences have the same realization graphs, and to certain addition operations we must first describe.

To keep our description mostly self-contained, we briefly recall some results from~\cite{Barrus16}. Recall that a \emph{split graph} is a graph whose vertex set may be partitioned into a clique and an independent set. For any split graph, we write the degree sequence as a ``splitted'' sequence $(p_2; p_1)$, where $p_1$ and $p_2$ are respectively the sublists containing degrees of vertices in the independent set and clique. (In our notation $p_2$ appears before $p_1$ because the vertices in the clique have degrees at least as large as those in the independent set; we will assume that the sublists $p_2$ and $p_1$ are each written in nonincreasing order.)

Tyshkevich~\cite{Tyshkevich00} defined a composition of degree sequences in the following way. If $|\pi|$ denotes the length of a list $\pi$ of integers, then for a splitted degree sequence $p=(p_2;p_1)$ and an arbitrary degree sequence $q$, the composition $p \circ q$ is formed by concatenating the following:
\begin{itemize}
    \item[(i)] the terms of $p_2$, each augmented by $|q|$,
    \item[(ii)] the terms of $q$, each augmented by $|p_2|$, and 
    \item[(iii)] the terms of $p_1$.
\end{itemize}
Observe that the resulting terms of $p \circ q$ appear in descending order. Note also that if $P$ and $Q$ are respectively realizations of the degree sequences $p$ and $q$, where the vertex set of $P$ is partitioned into an independent set $V_1$ and a clique $V_2$ in such a way that the vertices in $V_1$ and $V_2$ have degrees listed in $p_1$ and $p_2$, respectively, then $p \circ q$ is the degree sequence of the graph formed by taking the disjoint union of $P$ and $Q$ and adding an edge from each vertex of $Q$ to each vertex in $V_2$. We denote this graph by $(P,V_1,V_2) \circ Q$.

If the degree sequence $q$ in the discussion above is the degree sequence of a split graph, and in the realization $Q$ the vertex set has a partition $W_1,W_2$ into an independent set and clique, then $(P,V_1,V_2) \circ Q$ is a spit graph, and $p \circ q$ may be treated as a splitted sequence $(r_2; r_1)$ with the terms of $r_1,r_2$ corresponding to degrees of vertices in $V_1 \cup W_1$ and in $V_2 \cup W_2$, respectively. With this understanding, the operation $\circ$ is associative for both degree sequences and graphs.

\begin{figure}
    \centering
    \includegraphics[height=3.5cm]{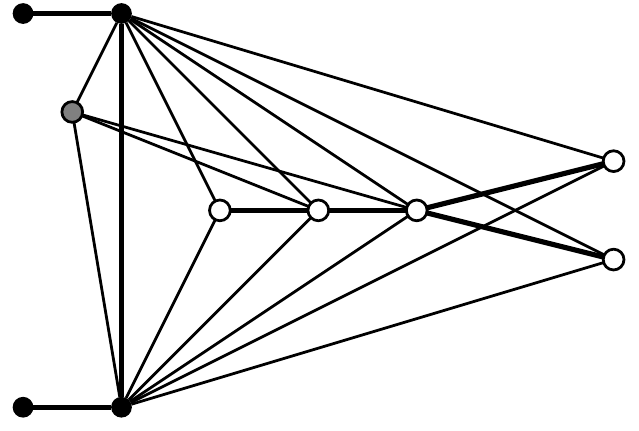}
    \caption{An example of the composition operation $\circ$}
    \label{fig: composition}
\end{figure}
In Figure~\ref{fig: composition} we illustrate the graph $(G_2,A_2,B_2) \circ  (G_1,A_1,B_1) \circ G_0$, where the graphs $G_0, G_1, G_2$ are realizations of the degree sequences $(0)$, $(3,2;1,1,1)$, and $(2,2;1,1)$, respectively. Here the vertices of $G_0$, $G_1$, and $G_2$ are respectively colored gray, white, and black. The sets $A_1,A_2$ are comprised of the vertices of degree $1$ in $G_1,G_2$, respectively, and the sets $B_1,B_2$ respectively contain the other vertices of $G_1,G_2$. Observe that the graph has degree sequence $(8,8,6,5,4,3,3,3,1,1)$, which equals $(2,2;1,1) \circ (3,2;1,1,1) \circ (0)$.

A degree sequence $d$ is \emph{decomposable} if $d = p \circ q$ for a splitted degree sequence $p$ and a degree sequence $q$, each of length at least 1. Otherwise, $d$ is said to be \emph{indecomposable}. In~\cite{Tyshkevich00} and earlier papers referred to therein, Tyshkevich showed the following.

\begin{thm}[\cite{Tyshkevich00}]
Every degree sequence $d$ may be expressed as a composition \begin{equation} \label{eq: decomposition} d=\alpha_1 \circ \dots \circ \alpha_k \circ d_0\end{equation} of indecomposable degree sequences, where each sequence $\alpha_i$ is a splitted degree sequence $(\beta_i; \gamma_i)$, and $d_0$. Moreover, this decomposition is unique.
\end{thm}

We refer to such an expression \eqref{eq: decomposition} as the \emph{Tyshkevich decomposition} of $d$.

The Tyshkevich decomposition gives us some understanding of the realization graph $\mathcal{G}(d)$. Let $G \Box H$ denote the Cartesian product of arbitrary graphs $G$ and $H$.

\begin{thm}[\cite{Barrus16}] \label{thm: Cartesian}
    If $d$ is a degree sequence having \[d=\alpha_1 \circ \dots \circ \alpha_k \circ d_0\] as its Tyshkevich decomposition, then \[\mathcal{G}(d) = \mathcal{G}(\alpha_1) \Box \cdots \Box \mathcal{G}(\alpha_k) \Box \mathcal{G}(d_0).\]
\end{thm}

Since a Cartesian product $G \Box H$ can be a complete graph if and only if one of $G,H$ is a complete graph and the other has a single vertex, it follows from Theorem~\ref{thm: Cartesian} that if $\mathcal{G}(d)$ is a complete graph, then all but possibly one of $\alpha_1,\dots,\alpha_k,d_0$ must have a single labeled realization.

Degree sequences having a unique labeled realization are known as \emph{threshold sequences}, and their realizations are \emph{threshold graphs}. (See~\cite{MahadevPeled95} for a book-length survey on properties of these graphs.) It is known that a degree sequence $d$ is a threshold sequence if and only if in the Tyshkevich decomposition of $d$, each indecomposable sequence has a single term. In this case each indecomposable sequence has the form (0) or (0;) or (;0). (See~\cite{Barrus12} for details.)

It follows that if $\mathcal{G}(d)$ is a complete graph, then we may write $d=t \circ \alpha \circ t'$, where both $t,t'$ are either empty (i.e., omitted) or threshold sequences, and $\alpha$ is an indecomposable degree sequence for which $\mathcal{G}(\alpha)$ is a complete graph. We now characterize such sequences $\alpha$.

Suppose that $\alpha$ is a degree sequence for which $\mathcal{G}(\alpha)$ is isomorphic to $K_n$, and let $R_1,\dots,R_n$ be the labeled realizations of $\alpha$. Since these realizations belongs to a clique of size $n$, Theorem~\ref{thm: clique implies dial} implies that a dial exists for these graphs. Adopting the same notation as in Section~\ref{sec: dials}, we let $u$ (respectively, $v$) be the vertex belonging to $n-1$ non-edges (respectively, $n-1$ edges) in each dial configuration; we let $w_1,\dots,w_n$ be the other dial vertices, labeled so that $uw_i$ is an edge in $R_i$ for each $i \in \{1,\dots,n\}$.

We claim that the graphs $R_i$ have no vertex other than those in $\{u,v,w_1,\dots,w_n\}$. Note that the alternating 4-cycles formed by the edges and non-edges of a dial configuration in any realization $R_i$ are sufficient to provide the 2-switches transforming $R_i$ into every other realization among $R_1,\dots,R_n$. Suppose now that $x$ is a vertex of $R_i$ not in $\{u,v,w_1,\dots,w_n\}$. Since the degree sequence $\alpha$ is indecomposable, it is known (see~\cite[Lemma 3.5]{Barrus12}) that $x$ belongs to an alternating 4-cycle. However, a 2-switch performed in $R_i$ on an alternating 4-cycle using $x$ would result in a realization of $\alpha$ not equal to any of $R_1,\dots,R_n$, contradicting the assumption that $\mathcal{G}(\alpha)$ has just these $n$ vertices. 

The need to prevent other ``unauthorized'' 2-switches gives us further restrictions. Fix $j \in \{1,\dots,n\}$. Suppose first that $u$ and $v$ are adjacent in $R_j$, and $i$ is an element of $\{1,\dots,n\}$ other than $j$. Note that if $w_i$ is adjacent to $w_j$ in $R_j$, then $[u,v:w_i,w_j]$ is an alternating 4-cycle in $R_j$, and performing the associated 2-switch in $R_j$ results in a realization in which $w_j$ is adjacent to both $u$ and $v$. This is a contradiction, since $R_1,\dots,R_n$ are the only realizations of $\alpha$. Hence for no $i \in \{1,\dots,n\}$ is $w_i$ adjacent to $w_j$. Moreover, since no 2-switch using edges and non-edges of the dial configuration changes the adjacency relationships among vertices in $\{w_1,\dots,w_n\}$, by varying $j$ in the argument above we conclude that $\{w_1,\dots,w_n\}$ must be an independent set. At this point the edges of each realization have been completely determined, and we verify that $\alpha$ is the degree sequence $\left(n,2,1^{(n)}\right)$. 

A similar argument shows that if $u$ and $v$ are not adjacent in $R_j$, then the vertices $w_1,\dots,w_n$ must be pairwise adjacent if $\mathcal{G}(d)$ is isomorphic to $K_n$. Here again the edges of $R_j$ and all other realizations have been completely determined; in this case $\alpha$ is the degree sequence $\left(n^{(n)},n-1,1\right)$.

A straightforward verification shows that both $\left(n,2,1^{(n)}\right)$ and $\left(n^{(n)},n-1,1\right)$ have exactly $n$ realizations, each of which is isomorphic to the appropriate graph shown in Figure~\ref{fig: G is Kn}, and the degree sequences have $K_n$ as their realization graph.
\begin{figure}
    \centering
    \includegraphics[height=1.4in]{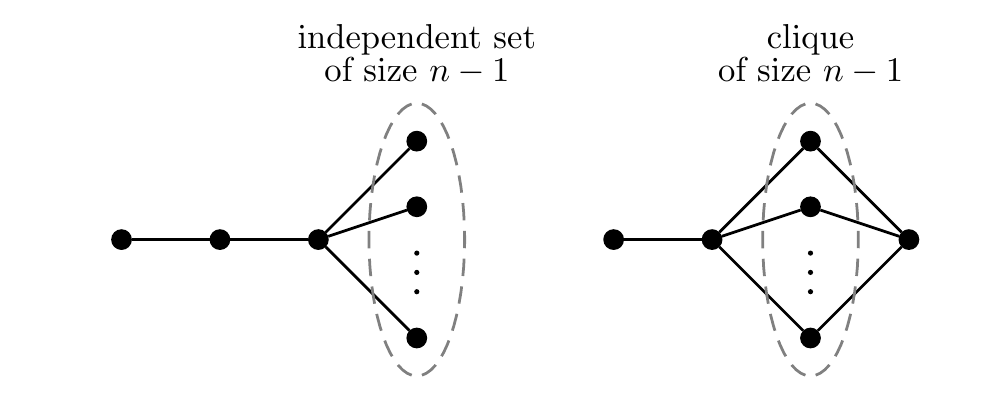}
    \caption{Graphs whose realization graphs are complete graphs}
    \label{fig: G is Kn}
\end{figure}

The discussion above proves our final result.

\begin{thm} \label{thm: Kn iff} 
For any $n \geq 4$ and any degree sequence $d$, the realization graph $\mathcal{G}(d)$ is a complete graph of order $n$ if and only if $d=t \circ \alpha \circ t'$, where each of $t,t'$ is either empty (i.e., omitted) or a threshold sequence, and $\alpha$ is $\left(n,2,1^{(n)}\right)$ or $\left(n^{(n)},n-1,1\right)$.
\end{thm}

\section*{Acknowledgments}
The authors wish to thank the anonymous referees for thoughtful comments that have improved the presentation of this paper.

\end{document}